\documentclass[12pt,a4paper]{article}
\usepackage[utf8]{inputenc}
\usepackage{amsmath}
\usepackage{amsthm}
\usepackage{amsfonts}
\usepackage{amssymb}
\usepackage{mathrsfs}

\usepackage{bbm}
\usepackage{graphicx}
\usepackage[all,cmtip]{xy}
\usepackage[bookmarks=true]{hyperref}
\usepackage[left=2.5cm,right=2.5cm,top=2cm,bottom=2cm]{geometry}

\theoremstyle{theorem}
\newtheorem{theorem}{Theorem}[section]
\newtheorem{proposition}[theorem]{Proposition}

\newtheorem{lemma}[theorem]{Lemma}

\newtheorem{claim}[theorem]{Claim}
\newtheorem{corollary}[theorem]{Corollary}

\newtheorem{question}[theorem]{Question}

\theoremstyle{definition}
\newtheorem{definition}[theorem]{Definition}
\newtheorem{remark}[theorem]{Remark}

\newcommand{\Vol}{\mathrm{Vol}}
\newcommand{\R}{\mathbb{R}}

\newcommand{\Q}{\mathbb{Q}}

\newcommand{\bb}{\mathbb}

\usepackage{cancel} 
\usepackage{enumerate}
\usepackage{color}
\usepackage{xcolor}
\usepackage{verbatim}
\usepackage{amsmath}
\usepackage{amssymb}
\usepackage{pifont}
\usepackage{ulem}
\usepackage{enumerate}
\allowdisplaybreaks[4]

\def\address#1#2{\begingroup
\noindent\parbox[t]{7.8cm}{%
\small{\scshape\ignorespaces#1}\par\vskip1ex
\noindent\small{\itshape E-mail address}%
\/: #2\par\vskip4ex}\hfill%
\endgroup}%

\title{Openness of uniformly valuative stability on the K\"ahler cone of projective manifolds}
\author{Yaxiong Liu}
\date{}

\begin{document}
      	
\maketitle

\begin{abstract}
	Assume that a projective variety is uniformly valuatively stable  with respect to a polarization.  We show that the projective variety is uniformly valuatively stable with respect to any polarization sufficiently close to the original polarization. 
The definition of uniformly valuatively stability in this paper is stronger than that given by Dervan and Legendre in \cite{D-L20}.
We also define the valuative stability for the transcendental K\"ahler classes. 
Our openness result can be extended to the K\"ahler cone of projective manifolds.
\end{abstract}

\section{Introduction}
Finding a canonical metric on K\"{a}hler manifolds is the central problem in K\"{a}hler geometry, especially finding a constant scalar curvature K\"ahler (cscK) metric.
The Yau-Tian-Donaldson (YTD) conjecture predicts that the existence of a cscK metric is equivalent to  algebro-geometric stability, the so-called \textit{K-polystability} due to Tian \cite{TG} and Donaldson \cite{Don02}, on a polarized manifold. 
When the automorphism group of a variety is discrete, K-polystability reduces to K-stability since the classical \textit{Futaki invariant} \cite{FA83} vanishes automatically.
Recently, the notion of K-(poly)stability of a polarized variety has played an important role in algebraic geometry, especially Fano setting.

The YTD conjecture is widely open in general.
There have been considerable strides on these ideas for the Fano case in recent years.
Chen-Donaldson-Sun \cite{CDS15} and Tian \cite{TG15} proved independently that K-polystability implies the existence of K\"ahler-Einstein metrics on Fano manifolds, solving this conjecture in the Fano case (also see \cite{CSW18}, \cite{DS16}, \cite{BBJ20}, \cite{Zh21} for other different methods). 
In the algebraic side, the theory has achieved substantial progress. 
The main breakthrough is due to Fujita \cite{Fuj19 a} (also Fujita and Odaka \cite{FO18}) and Li \cite{Li17}, which re-interprets K-stability in terms of valuations by the algebraic invariant, the so-called \textit{$\delta$-invariant} \cite{FO18} or \textit{$\beta$-invariant} \cite{Fuj19 a}, \cite{Li17}.
One can test K-stability of a Fano variety by computing its $\delta$-invariant or $\beta$-invariant. This is the so-called Fujita-Li criterion.

People study K-stability of Fano varieties from the viewpoint of birational geometry.
An almost complete theory of K-stability of Fano varieties is established.
From this powerful theory, one can construct a desirable moduli space of K-stable Fano varieties, the so-called \textit{K-moduli space}. There are many important works along these lines, due to Xu, Liu, Zhuang, Blum, etc. (see \cite{BX19}, \cite{CP21}, \cite{BLX19}, \cite{Xu20}, \cite{ABHLX20}, \cite{BHLLX20}, \cite{XZ20}, etc.).
We refer the reader to an excellent survey \cite{Xu20a} for the algebraic theory of K-stability of Fano varieties.
Very recently, Liu-Xu-Zhuang \cite{LXZ21} proved that the K-moduli space is proper by solving two profound and challenging conjectures, the so-called Higher Rank Finite Generation conjecture and Optimal Destabilization conjecture.
As an application of those conjectures, they also show that K-stability is equivalent to uniform K-stability for a log Fano pair with the discrete automorphism group.
Moreover, their argument also holds for the non-discrete automorphism group.
The Fujita-Li criterion for K-stability of Fano varieties has played an essential role in all of these developments.

To study K-stability of polarized varieties, the next step is to develop the valuative criterion in the polarized case. 
The original definition of K-stability involves ${\bb C}^*$-degenerations of a polarized variety, the so-called \textit{test configurations}.
Donaldson \cite{Don02} associates a numerical invariant to each test configuration, the so-called \textit{Donaldson-Futaki invariant}. 
K-stability means that this invariant is always positive.
By works of Boucksom, Jonsson, etc (see \cite{BHJ17}, \cite{BJ18}), we can identify a test configuration with a finite generated $\bb Z$-filtration on the section ring of the polarization.

For any valuation, one can associate a filtration to this valuation. When this filtration is finitely generated, such a valuation is called a \textit{dreamy valuation}.
A valuation is called a \textit{divisorial valuation} if it is induced by a prime divisor.
A divisor is called a \textit{dreamy divisor} if the corresponding divisorial valuation is dreamy. 
Dervan and Legendre \cite{D-L20} define a new $\beta$-invariant for polarised varieties, 
which generalizes Fujita's original $\beta$-invariant, by computing the Donaldson-Futaki invariant of the test configuration associated with a dreamy divisor.
They showed that K-stability over integral test configurations is equivalent to valuative stability over dreamy divisors.
Here an integral test configuration means that its central fiber is integral.
It gives an expectation to establish the Fujita-Li criterion in the polarised case.

Unfortunately, examples in \cite{ACGTF08} show that positivity of the Donaldson-Futaki invariant for algebraic test-configurations may not be enough to ensure the existence of a cscK metric.
A stronger notion, the so-called \textit{uniform K-stability}, is introduced by the thesis \cite{Sze06}, and deeply developed in \cite{BHJ17} and \cite{Der16}, which becomes a new candidate for the stability criterion of the existence of a cscK metric.
When the automorphism group of manifolds is discrete, the uniform YTD conjecture states that the uniform K-stability is equivalent to the existence of a cscK metric. 
Very recently, 
Li \cite{Li20} proved the existence of cscK metrics under the condition of  \textit{uniform K-stability for model filtration}, which is stronger than the original uniform K-stability. Moreover, his approach also holds when the automorphism group is non-discrete.

A basic question about uniform stability is whether it is preserved under small perturbations of the polarization or not.
This question is motivated by a classical result of LeBrun-Simanca \cite{LS94}, in which they establish openness results for perturbations of cscK metrics. 
Fujita \cite{Fuj19 b} proved the openness of uniform K-stability for the log canonical and log anti-canonical polarization.
Note that Fujita's result requires that the base variety can have bad singularities, the so-called \textit{demi-normal pair} (see \cite{Kol13} or \cite{Fuj19 b}).
Zhang \cite{Zh21a} proved that the valuative stability threshold ($\delta$-invariant) is continuous on the big cone of Fano manifolds.
Thus, the openness of uniformly valuative stability holds for Fano manifolds. 

In this paper, we establish openness of uniformly valuative stability  for general projective varieties. 
This gives an affirmative answer to the above question for uniformly valuative stability.
Note that our definition of uniformly valuative stability is stronger than that given by Dervan and Legendre in \cite{D-L20}, see Definition \ref{def of valuative stability} and Remark \ref{D-L definition}.

Our main theorem is
\begin{theorem}[see Theorem \ref{UVs open}]
\label{thm 1.1}
For a normal projective variety $X$, the uniformly valuative stability locus 
\begin{equation*}
	{\rm UVs}:=\{[L]\in{\rm Amp}(X)|(X,L)\text{ is uniformly valuatively stable}\}
\end{equation*}
is an open subcone of the ample cone ${\rm Amp(X)}$.
\end{theorem}
Together with LeBrun-Simanca's openness, our result fits the expectation of YTD conjecture. 

A main difficulty of Theorem \ref{thm 1.1} is to control the difference of the derivative part in the expression of $\beta$-invariant for two nearby ample divisors.
It is hard to control the difference for all prime divisors in general.
In addition, the log discrepancy has no control generally.
By considering the derivative part of $\beta$-invariant together with the log discrepancy, 
we obtain a partial control of $\beta$-invariant (see Theorem \ref{lower bdd theorem}), which is enough to show our main theorem.

As an immediate application of main theorem, we obtain
\begin{corollary}[see Theorem \ref{cont of u.v.s.t.}]
\label{cor 1.2}
    For a normal projective variety $X$,
	the uniformly valuative stability threshold
	\begin{equation*}
		{\rm Amp}(X)\ni L\mapsto \zeta(L)\in{\bb R}
	\end{equation*}
	is continuous on the ample cone ${\rm Amp}(X)$ $($see Definition \ref{u.v.s.t.} for $\zeta(L)$$)$.
\end{corollary}
The invariant $\zeta$ is motivated by $\delta$-invariant since $\delta-1$ can be viewed as the stability threshold (in the sense of Definition \ref{u.v.s.t.}) of the original $\beta$-invariant.
Corollary \ref{cor 1.2} gives a similar result with Zhang \cite{Zh21a} for projective varieties.
According to the expression of $\beta$-invariant, we do not have a canonical formulation to define its corresponding $\delta$-invariant for polarised varieties.
Studying the invariant $\zeta$ is a good candidate to test valuative stability.

The definition of cscK metrics does not need a polarization. 
In \cite{DR17} and \cite{SD18}, they define independently the notion of K-stability for the transcendental K\"ahler classes.
It is natural to extend the valuative stability to any K\"ahler class of compact K\"ahler manifolds (see Definition \ref{def of valuative stability for transc}).

Due to some well-known results about analytic geometry, it is straightforward to see that our argument for the algebraic class can also work for the K\"ahler class on projective manifolds. 
We state it as follows,
\begin{theorem}[see Theorem \ref{UVs open Kahler}]
	For a projective manifold $X$, the uniformly valuative stability locus 
	\begin{equation*}
		\widehat{\rm UVs}:=\{\alpha\in{\cal K}|(X,\alpha)\text{ is uniformly valuatively stable}\}
	\end{equation*}
	is an open subcone of the K\"ahler cone $\cal K$.
\end{theorem}

This article is organized as follows.
In Section \ref{Preliminaries}, we recall basic theories of the volume and the positive intersection product. 
Moreover, we collect some well-known properties and some useful theorems.
In Section \ref{Valuative stability}, we recall these definitions of the $\beta$-invariant and the valuative stability of polarized varieties. Moreover, we formulate our main theorem.
In Section \ref{Proof of openness of valuative stability}, we finish the proof of the main theorem.
In section \ref{Continuity of uniformly valuative stability threshold},
    as an application of the main theorem, we prove the continuity of uniformly valuative stability threshold.
In Section \ref{Valuative stability for transcendental classes}, 
we compile some basic theories of the analytic volume and positive intersection product. Moreover, we state some related facts and extend our result to the K\"ahler cone of projective manifolds.
In Section \ref{Further questions}, we propose some interesting further questions.

\subsection*{Acknowledgement}
	I would like to thank my supervisor Akito Futaki for constant help, his guidance and teaching over many years. 
	I am grateful to Jian Xiao for telling me the differentiability result of Witt Nystr\"om \cite{WN19} and Kewei Zhang for helpful discussions.  
	I wish to thank the anonymous referees for their many helpful comments.
	
\subsection*{Notation}
We work throughout over the complex number $\bb C$.
A \textit{variety} is always assumed to be a connected, reduced, separated and of finite type scheme over $\mathrm{Spec}\ \mathbb{C}$.
Unless we say specifically, in this article, we fix $X$ as an $n$-dimensional normal projective $\mathbb{Q}$-Gorenstein varieties, and fix all divisors as Cartier divisors.
For the convenience of writing, we do not distinguish between divisors and line bundles.

\section{Preliminaries}
\label{Preliminaries}
\subsection{Volume function}
In this subsection, we review some of the standard facts on the volume function.

The N\'eron-Severi space $N^1(X)$ is the real vector space  of  numerical equivalence classes of $\R$-Cartier divisors on $X$.
In general, the N\'eron-Severi space is denoted by $N^1(X)_{\bb R}$. But for simplicity, we denote it by $N^1(X)$. 
For any Cartier divisor $D$, the \textit{volume} of $D$ is defined to be
\begin{equation}
\label{def of volume}
	\Vol(D)=\limsup_{m\rightarrow\infty}\frac{h^0(X,mD)}{m^n/n!}.
\end{equation}
For any natural number $a>0$, we have
\begin{equation*}
	\Vol(aD)=a^n\Vol(D).
\end{equation*}
It follows that the volume for any $\Q$-Cartier divisor $D$ to be
\begin{equation*}
	\Vol(D)=\frac{1}{a^n}\Vol(aD),
\end{equation*}
for some $a\in{\bb N}$, such that $aD$ is Cartier divisor. This is independent of the choice of $a$.
The volume of a $\Q$-Cartier divisor depends only on its numerical equivalence class.
Thus, the volume function can be descended to $N^1(X)_{\Q}$.
Then the
volume function extends continuously to $N^1(X)$.
The volume satisfies the homogeneous property, i.e. 
\begin{equation*}
	\Vol(aD)=a^n\Vol(D).
\end{equation*}
for any $a>0$ and any $D$ in $N^1(X)$.

We recall some definitions of positivity of $\bb R$-divisors. 
An $\bb R$-divisor $D$ in $N^1(X)$ is called \textit{nef} if the intersection number $L\cdot C$ is nonnegative for any curve $C$ on $X$. 
The volume of a nef $\bb R$-divisor $D$ is equal to the top self-intersection number $D^n$.
All nef classes in $N^1(X)$ form a convex cone, called the \textit{nef cone}, denoted by ${\rm Nef}(X)$, whose interior is called the \textit{ample cone}, denoted by ${\rm Amp}(X)$. 
An $\R$-divisor $D$ in $N^1(X)$ is called \textit{big} if 
\begin{equation*}
	\Vol(D)>0.
\end{equation*}
All big classes in $N^1(X)$ form a convex open cone, called the \textit{big cone}, denoted by ${\rm Big}(X)$,  whose closure is called the \textit{pseudo-effective $($psef for short$)$ cone}.
For any two big $\R$-divisors $D$ and $B$, one can obtain
\begin{equation*}
	\Vol(D+B)\geq\Vol(D).
\end{equation*}
For more details of the volume function, we refer to the standard reference \cite{L04}.

\subsection{Positive intersection product}
\label{Positive intersection product}
In this subsection, we present some preliminaries about the positive intersection product. 
We follow the notions of \cite{BFJ09}.
References to this subsection are \cite{BDPP12}, \cite{BFJ09}, \cite[Section 7]{Sh03}, \cite{D-F20}.

In general, the volume of a big divisor is not equal to its top self-intersection number. 
But it can be computed as the movable intersection number (see \cite[Chapter 11]{L04}) by Fujita's approximation theorem.
In other words, for any big divisor $D$, let $\pi_m:X_m\rightarrow X$ be the resolution of base locus $\mathfrak{b}(|mD|)$ with the exceptional divisor $E_m$ and set $D_m=\pi_m^*D-\frac{1}{m}E_m$,
then
\begin{equation*}
	\Vol(D)=\lim_{m}D_m^{n}.
\end{equation*} 
To compute the volume of a big divisor, in \cite{BFJ09} the authors introduce a valid notion, the so-called \textit{positive intersection product}.
Next we recall the notion (also see \cite{BFJ09}, \cite{BDPP12} for details).

Recall that the Riemann-Zariski space of $X$ is the locally ringed space defined by
\begin{equation*}
	\mathfrak{X}:=\varprojlim_{\pi} X_{\pi}\,,
\end{equation*}
where $X_{\pi}$ runs over all birational models of $X$ with the birational morphism $\pi:X_{\pi}\rightarrow X$.
Here the projective limit is taken in the category of locally ringed spaces.
We do not use the theory of Riemann-Zariski spaces in an essential way in this paper. 
We refer to \cite{ZS75,Va00} for more discussions on the structure of this space.
\begin{definition}[{\cite[Definition 1.1]{BFJ09}}]
For any integer $0\leq p\leq n$,
	\begin{enumerate}[-]
		\item the space of $p$-codimensional Weil classes on the Riemann-Zariski space $\mathfrak{X}$ is defined as 
		    \begin{equation*}
		    	N^p(\mathfrak{X})
		    	:=\varprojlim_{\pi}N^p(X_\pi)\,,
		    \end{equation*}
		    with arrows defined by push-forward, where $N^p(X_\pi)$ is the real vector space of numerical equivalence classes of codimension $p$-cycles (see \cite[Chapter 19]{Fu97}). 
		 \item the space of $p$-codimensional Cartier classes on $\mathfrak{X}$ is defined as
		     \begin{equation*}
		     	CN^p(\mathfrak{X}):=\varinjlim_{\pi}N^p(X_\pi)\,,
		     \end{equation*}
		    with arrows defined by pullback.
	\end{enumerate}
\end{definition}
By definition, a Weil class $\alpha$ in $N^p(\mathfrak{X})$ is given by its \textit{incarnations} $\alpha_\pi$ in $N^p(X_\pi)$ on each smooth birational model of $X$, 
satisfying 
$$\nu_{*}(\alpha_{\pi'})=\alpha_{\pi''}$$
for any birational morphism $\nu:X_{\pi'}\rightarrow X_{\pi''}$ with $\pi'=\pi''\circ \nu$.

Further for each $\pi$, given a class $\alpha$ in $N^p(X_\pi)$, one can extend it to a Cartier class by pullback of it.
Thus, we have the natural injection
\begin{equation*}
	N^p(X_\pi)\hookrightarrow CN^p(\mathfrak{X}),
\end{equation*}

When $p=1$, we refer to the space $CN^1(\mathfrak{X})$ as the \textit{N\'eron-Severi space} of $\mathfrak{X}$. Its elements are the so-called Shokurov's b-divisors.

In the sequel, we use the notation $\alpha\geq0$ for a psef class $\alpha$ in $N^p(X)$ (see \cite{Fu97}).
We consider positive Cartier classes in $\mathfrak{X}$.
For a birational morphism $\nu:V'\rightarrow V$, a class $\alpha$ in $N^1(V)$ is nef (resp. psef, big) if and only if $\nu^*\alpha$ is nef (resp. psef, big).
Therefore, one can extend these definitions to the Riemann-Zariski space.

\begin{definition}[{\cite[Definition 1.6]{BFJ09}}]
	A Cartier class $\alpha$ in $CN^1(\mathfrak{X})$ is called nef (resp. psef, big) if its incarnation $\alpha_\pi$ is nef (resp. psef, big) for some $\pi$.
\end{definition}
On a smooth projective variety $V$, for any $p$-classes $\alpha_1,\ldots,\alpha_p$ in $N^1(V)$, the intersection product $\alpha_1\cdots\alpha_p$ belongs to $N^p(V)$ (see \cite{Fu97}).
Further for any birational morphism $\nu:V'\rightarrow V$, one has 
$\nu^*\alpha_1\cdots\nu^*\alpha_p
=\nu^*(\alpha_1\cdots\alpha_p)
$,
see \cite[Chapter 19]{Fu97}.
One can define the intersection product of $p$-Cartier classes $\alpha_1,\ldots,\alpha_p$ in $CN^1(\mathfrak{X})$, 
which have a common determination $X_\pi$, as the Cartier class in $CN^p(\mathfrak{X})$ determined by $\alpha_{1,\pi}\cdots\alpha_{p,\pi}$. 
\begin{definition}[{\cite[Definition 2.5]{BFJ09}}]
	For any big classes $\alpha_1,\ldots,\alpha_p$ in  $CN^1(\mathfrak{X})$, their \textit{positive intersection product} 
	\begin{equation*}
\langle\alpha_1\cdots\alpha_p\rangle\in N^p(\mathfrak{X})
	\end{equation*}
	 is defined as the least upper bound of the set of classes
	\begin{equation*}
		(\alpha_1-D_1)\cdots(\alpha_p-D_p) \in N^p(\mathfrak{X})
	\end{equation*}
	where $D_i$ is an effective Cartier $\Q$-divisor on $\mathfrak{X}$ such that $\alpha_i-D_i$ is nef.
\end{definition} 
\begin{remark}
\label{analytic pos inters}
	In \cite[Theorem 3.5]{BDPP12}, the authors give an analytic definition of the positive intersection product (they call it as the movable intersection product) for K\"ahler manifolds.
	For any \textit{big classes} $\alpha_1,\ldots,\alpha_p$ on the K\"ahler manifold $V$, in which the big class means that each $\alpha_j$ can be represented by a \textit{K\"ahler current} $T$, i.e. a closed positive $(1,1)$-current $T$ such that $T\geq\theta\omega$ for some smooth Hermitian metric $\omega$ and a small constant $\theta>0$,
	one defines
	\begin{equation*}
		\langle\alpha_1\cdots\alpha_p\rangle
		:=\sup_{T_j,\pi}\pi_{*}(\gamma_1\wedge\ldots\wedge\gamma_p)
	\end{equation*}
	where $T_j\in\alpha_j$ is a K\"ahler current with logarithmic poles, i.e. there is a modification $\pi_j:V_j'\rightarrow V$ such that $\pi_j^*T_j=[E_j]+\gamma_j$ for some effective $\Q$-divisor $E_j$ and closed semi-positive form $\gamma_j$. 
	Here we take a common resolution $\pi:V'\rightarrow V$, and write
	\begin{equation*}
		\pi^*T_j=[E_j]+\gamma_j.
	\end{equation*}
\end{remark}
\begin{definition}[{\cite[Definition 2.10]{BFJ09}}]
	For any psef classes $\alpha_1,\ldots,\alpha_p$ in  $CN^1(\mathfrak{X})$,
	their positive intersection product
	\begin{equation*}
     \langle\alpha_1\cdots\alpha_p\rangle\in N^p(\mathfrak{X})
	\end{equation*}
	is defined as the limit
	\begin{equation*}
		\lim_{\varepsilon\rightarrow0^+} \langle(\alpha_1+\varepsilon\gamma)\cdots(\alpha_p+\varepsilon\gamma)\rangle,
	\end{equation*}
	where $\gamma$ in $CN^1(\mathfrak{X})$ is any big Cartier class.
\end{definition}
This definition is independent of the choice of the big class $\gamma$ (see \cite[Definition 2.10]{BFJ09}).

For any big $\bb R$-divisor $D$ in $N^1(X)$, we have
\begin{equation*}
	\Vol(D)=\langle D^n\rangle,
\end{equation*}
also see \cite[Definition 3.2]{BDPP12} or \cite[Definition 1.17]{BEGZ10} for an analytic definition.

An interesting fact about the volume function on the big cone, due to Boucksom-Favre-Jonsson \cite{BFJ09}, is stated as follows,
\begin{theorem}[{\textup{\cite[Theorem A]{BFJ09}}}].
\label{differentiable of volume}
	The volume function is $C^1$-differentiable on the big cone of $N^1(X)$. 
	If $\alpha\in N^1(X)$ is big and $\gamma\in N^1(X)$ is arbitrary, then
	\begin{equation}
	\label{diff of vol}
		\frac{d}{dt}\Big|_{t=0}\Vol(\alpha+t\gamma)
		=n\langle\alpha^{n-1}\rangle\cdot\gamma.
	\end{equation}
\end{theorem}
We collect some facts about the positive intersection product as follows, for using later,
\begin{proposition}[{\textup{\cite[Proposition 2.9, Corollary 3.6]{BFJ09}}}]
\label{cont of positive intersection product}
\
	\begin{enumerate}[$(\rm i)$]
		\item The positive intersection product
	           is symmetric, homogeneous of degree $1$, and super-additive in each variable.
	          Moreover, it is continuous on the $p$-fold product of the big cone of $CN^1(\mathfrak{X})$.
	    \item For any psef class $\alpha$ in $CN^1(\mathfrak{X})$, one obtain
	        \begin{equation*}
	        	\langle\alpha^n\rangle
	        	=\langle\alpha^{n-1}\rangle\cdot\alpha.
	        \end{equation*}
	\end{enumerate}
\end{proposition}
\begin{remark}
	In general, the positive intersection product is not multilinear, see \cite[Definition 1.17]{BEGZ10} for an analytic explanation. 
\end{remark}

\section{Valuative stability}
\label{Valuative stability}
In this section, we review the $\beta$-invariant given by \cite{D-L20} and state the definition of valuative stability.

In \cite{D-L20}, Dervan and Legendre compute the Donaldson-Futaki invariant of the test configuration associated to a dreamy divisor for a polarised variety and obtain a new numerical invariant, which generalizes Fujita's original $\beta$-invariant.
Then they show that valuative stability for dreamy divisors is equivalent to $K$-stability for integral test configurations. 
Here an integral test configuration means that its central fiber is integral.
In this paper, we do not involve the explicit definition of K-stability, refer to \cite{D-L20}, \cite{BHJ17}, \cite{Don02}.
~\\

Let $(X,L)$ be a polarized variety, let $\pi: Y\rightarrow X$ be a surjective birational morphism.
\begin{definition}
	A prime divisor $F\subset Y$ for some birational model $Y$ over $X$ is called a \textit{prime divisor over} $X$.
	Denote by ${\rm PDiv}_{/X}$ the set of all prime divisors over $X$.
\end{definition}
One can view $F$ as a divisorial valuation ${\rm ord}_F$ on $X$, defined on the function field of $X$.
In particular, we can always assume that $Y$ is smooth by taking a resolution of singularities. 
Since the information of the valuation associated to $F$, which we are interested in, does not change under the resolution of singularities,
see \cite[Remark 2.23]{KM98}.

\begin{definition}
	For any $F$ in ${\rm PDiv}_{/X}$, the \textit{log discrepancy} $A_X(F)$ is defined to be
\begin{equation*}
	A_X(F):=1+{\rm ord}_F(K_Y-\pi^*K_X).
\end{equation*}
\end{definition}
Note that the log discrepancy is well-defined, since we always assume that the canonical divisor $K_X$ is $\mathbb{Q}$-Cartier. 

For any prime divisor $F$ over $X$, one can define a subspace $H^0(X,mL-xF)\subset H^0(X,mL)$ by the identifications
\begin{equation*}
H^0(X,mL-xF):=H^0(Y,m\pi^*L-xF)\subset H^0(Y,m\pi^*L)\cong H^0(X,mL).
\end{equation*}
For any ample divisor $L$, one defines the \textit{slope} of $(X,L)$ to be
\begin{equation*}
	\mu(L):=\frac{-K_X\cdot L^{n-1}}{L^n}.
\end{equation*}

For any $F$ in ${\rm PDiv}_{/X}$, Dervan and Legendre define
\begin{equation}
\label{beta def}
	\beta_L(F)
	:=A_X(F)\Vol(L)+n\mu(L)\int_0^{+\infty}\Vol(L-xF)dx
	+\int_0^{+\infty}\Vol^{\prime}(L-xF)_{\cdot}K_Xdx,
\end{equation}
where 
\begin{equation*}
	\Vol(L-xF):=\Vol(\pi^*L-xF),
\end{equation*}
and
\begin{equation*}
\label{differ of vol} 
	\Vol^{\prime}(L-xF)_{\cdot}K_X
	:=\frac{d}{dt}\Big|_{t=0}\Vol(\pi^*L-xF+t\pi^*K_X).
\end{equation*}
For simplicity, we always omit $\pi^*$ in the above notations.
It follows from Theorem \ref{differentiable of volume} that the notation $\Vol^{\prime}(L-xF)_{\cdot}K_X$ is well-defined for any $L$ in ${\rm Big}(X)$ and $F$ in ${\rm PDiv}_{/X}$.
It is straightforward that $\beta_L(\cdot)$ depends only on the numerical equivalence class of $L$. 

There are three numerical invariants on the space of prime divisors over $X$.
Roughly speaking, these can be viewed as norms.
For any $F$ in ${\rm PDiv}_{/X}$, we set
\begin{equation*}
	S_L(F)
	:=\int_0^{+\infty}\Vol(L-xF)dx,              
\end{equation*}
and
\begin{equation*}
	j_L(F)
	:=\Vol(L)\tau_L(F)-S_L(F).
\end{equation*} 
where $\tau_L(F)$ is the \textit{pseudo-effective threshold} of $F$ with respect to $L$, defined by
\begin{equation*}
	\tau_L(F)
	:=\sup\{x\in{\bb R} |\Vol(L-xF)>0 \}.
\end{equation*}
Note that our notation $S_L(F)$ is different from the usual one, which is equal to $S_L(F)/\Vol(L)$. 
But just for convenience, we use this notation.

\begin{lemma}
\label{relation of norms}
When $L$ is ample,
	for any prime divisor $F$, $\tau_L(F)$, $S_L(F)$ and $j_L(F)$ have the following relations
\begin{equation}
\label{rela of tau and j}
	\frac{1}{n+1}\Vol(L)\tau_L(F)\leq j_L(F)
	\leq\frac{n}{n+1}\Vol(L)\tau_L(F),
\end{equation}
and
\begin{equation}
\label{rela of tau and S}
	\frac{1}{n+1}\Vol(L)\tau_L(F)\leq S_L(F) \leq
	\frac{n}{n+1} \Vol(L)\tau_L(F).
\end{equation}
\end{lemma}
The invariant $j_L(\cdot)$ can be viewed as a norm corresponding to non-Archimedean functional $J^{\rm NA}$ 
and $S_L(\cdot)$ corresponds to $I^{\rm NA}-J^{\rm NA}$, see \cite[Section 2]{D-L20}, \cite{BJ18} and \cite[Section 7.2]{BHJ17}.
The proof of this lemma is essentially same as that of Fujita \cite{Fuj19} in Fano $L=-K_X$ case, also see \cite[Proposition 3.11]{BJ20}.
\begin{proof}[Proof of Lemma \ref{relation of norms}]
	We only need to show (\ref{rela of tau and S}). 
	The first inequality of (\ref{rela of tau and S}) is given by the concavity of the volume function, which gives
	\begin{equation*}
		\Vol(L-xF)\geq\Vol(L)\left(\frac{x}{\tau_L(F)} \right)^n.
	\end{equation*}
	It follows that 
	\begin{equation*}
		S_L(F)\geq\frac{1}{n+1}\Vol(L)\tau_L(F).
	\end{equation*}
	The second inequality is proved in \cite[Proposition 2.1]{Fuj19} (In \cite{Fuj19}, $L=-K_X$, but this condition is not used in the proof). 
\end{proof}


For any $L$ in ${\rm Amp}(X)$, we define two numerical invariants:
\begin{equation*}
	s(L):=\sup\{s\in\mathbb{R}|-K_X-sL \text{ is ample}\},
\end{equation*}
and
\begin{equation*}
	\tilde{s}(L):=\inf\{s\in\mathbb{R}|K_X+sL \text{ is ample}\}.
\end{equation*}
By definitions of $s(L)$ and $\tilde{s}(L)$, we have $\mu(L)\geq s(L)$ and $\mu(L)\leq\tilde{s}(L)$.
Indeed, if one assume that $-K_X-\mu(L)L$ is ample, then
\begin{align*}
	0<(-K_X-\mu(L)L)\cdot L^{n-1}
	=&(-K_X\cdot L^{n-1}-\mu(L)L^n)               \\
	=&\left(\frac{-K_X\cdot L^{n-1}}{L^n}-\mu(L) \right)L^n.
\end{align*}
This leads to a contradiction. It follows that $\mu(L)\geq s(L)$.
Another one is similar.

We state the following useful lemma (see \cite[Corollary 3.11]{D-L20}), and use this lemma repeatedly later in the article,
\begin{lemma}
\label{integration by part lemma}
For any big divisor $L$ in $N^1(X)$ and any prime divisor $F$ over $X$, we have
\begin{equation}
	\int_0^{\tau_L(F)}n\langle(L-xF)^{n-1}\rangle {\cdot}Ldx
	=(n+1)\int_0^{\tau_L(F)}\Vol(L-xF)dx.
\end{equation}
\end{lemma}
This Lemma is due to \cite[Corollary 3.11]{D-L20}. Its proof is a standard computation by using integration by part, left to the reader.
By this lemma,
we can re-write $\beta$ as
\begin{flalign}
	\beta_L(F)
	=&A_X(F)\Vol(L)+(n\mu(L)-(n+1)s(L))S_L(F) \nonumber  \\
	&-\int_0^{+\infty}\Vol^{\prime}(L-xF)_{\cdot}(-s(L)L-K_X)dx,   \label{s rep of beta}
\end{flalign}
or
\begin{flalign}
	\beta_L(F)
	=&A_X(F)\Vol(L)+(n\mu(L)-(n+1)\tilde{s}(L))S_L(F)  \nonumber \\
	&+\int_0^{+\infty}\Vol^{\prime}(L-xF)_{\cdot}(\tilde{s}(L)L+K_X)dx.  \label{s-tilde rep of beta}
\end{flalign}
\begin{definition}
\label{def of valuative stability}
	For any $L\in{\rm Amp}(X)$, $(X,L)$ is called 
	\begin{enumerate}[(i)]
	\item \textit{valuatively semistable} if 
          \begin{equation*}
          	\beta_L(F)\geq0
          \end{equation*}
          for any prime divisor $F$ over $X$;
    \item \textit{valuatively stable} if
          \begin{equation*}
          	\beta_L(F)>0
          \end{equation*}
          for any non-trivial prime divisor $F$ over $X$, in which the non-trivial prime divisor $F$ means that the divisorial valuation associated to $F$ is non-trivial;
	\item \textit{uniformly valuatively stable} if there exists an $\varepsilon_L>0$ such that
	      \begin{equation}
	    \label{uniform valuative stable}
		      \beta_L(F)\geq\varepsilon_LS_L(F)
	    \end{equation}
	      for any prime divisor $F$ over $X$.
	\end{enumerate}
\end{definition}
\begin{remark}
\label{D-L definition}
    \begin{enumerate}[$(\rm i)$]
    	\item Note that in \cite{D-L20}, valuative stability means that $\beta_L$ satisfies the demanded inequality for all dreamy divisors (see \cite[Definition 2.6]{D-L20}).
    If $\beta_L$ is nonnegative for all prime divisors over $X$,  it is called \textit{strongly valuatively semistable} in \cite{D-L20}.
   \item 
   In \cite{D-L20} the authors use the norm $j_L(\cdot)$ to define uniformly valuative stability.
   By Lemma \ref{relation of norms}, $j_L$ and $S_L$ are equivalent. 

    \end{enumerate}
\end{remark}

In this paper, we are interested in the openness of uniformly valuative stability. Our main theorem is stated as follows,
\begin{theorem}
\label{UVs open}
The uniformly valuative stability locus 
\begin{equation*}
	{\rm UVs}:=\{[L]\in{\rm Amp}(X)|(X,L)\text{ is uniformly valuatively stable}\}
\end{equation*}
is an open subcone of the ample cone ${\rm Amp(X)}$.
\end{theorem}

\section{Proof of openness of valuative stability}
\label{Proof of openness of valuative stability}
In this section, we give a proof of Theorem \ref{UVs open}.

We first give a rough idea of setup:
Fix an ample $\bb R$-divisor $L$, which is uniformly valuatively stable, and choose a constant $\varepsilon_L>0$ such that
\begin{equation*}
	\beta_L(F)\geq\varepsilon_LS_L(F)
\end{equation*}
for any prime divisor $F$ over $X$.
Our goal is to show that there exists a small open neighbourhood $U$ of $L$ in ${\rm Amp}(X)$ such that, for any $L'$ in $U$ there is a constant $\delta_{L'}>0$ satisfying
\begin{equation*}
	\beta_{L'}(F)\geq\delta_{L'} S_{L'}(F)
\end{equation*}
for all prime divisor $F$ over $X$.

To define such an open neighbourhood of $L$, we fix any norm $\|\cdot\|$ on $N^1(X)$ and define an open subset
\begin{equation*}
	U_{\varepsilon}:=\{
	L'\in{\rm Amp}(X)\ |\ \|L'-L\|<\varepsilon
	 \}. 
\end{equation*}
If necessary, we shrink this neighbourhood, i.e. shrink $\varepsilon$.

It suffices to prove following these two estimates

\begin{equation}
\label{diff of beta}
	\beta_{L'}(F)-\beta_{L}(F) \geq -f(\varepsilon)S_{L'}(F)
\end{equation}
and
\begin{equation}
	S_L(F)\geq s^{-}(\varepsilon)S_{L'}(F)  
\end{equation}
for any prime divisor $F$ over $X$, 
where 
$f:{\bb R}^+\rightarrow{\bb R}^+$ and $s^{-}:{\bb R}^+\rightarrow{\bb R}^+$ are continuous functions with  $f(\varepsilon)\rightarrow0$ and $s^{-}(\varepsilon)\rightarrow1$ as $\varepsilon \rightarrow0$.
Indeed,
\begin{flalign}
	\beta_{L'}(F)
	=&\beta_L(F)+\beta_{L'}(F)-\beta_L(F)    \nonumber            \\
	\geq&\varepsilon_LS_L(F)-f(\varepsilon)S_{L'}(F)    \nonumber              \\
	\geq&\left(\varepsilon_L s^-(\varepsilon)-f(\varepsilon)\right)S_{L'}(F).        \label{idea of uvs}
\end{flalign}

\begin{lemma}
\label{S estimate lemma}
	For any $L$ in ${\rm Amp}(X)$, there exists a small constant $\varepsilon>0$, such that for any $L'$ in $U_{\varepsilon}$ satisfying the following inequality 
	\begin{equation*}
		s^-(\varepsilon)S_{L'}(F)
		\leq S_L(F)\leq
		s^+(\varepsilon)S_{L'}(F)
	\end{equation*}
	for any $F$ in ${\rm PDiv}_{/X}$, where $s^-:{\bb R}^+\rightarrow{\bb R}^+$ and $s^{+}:{\bb R}^+\rightarrow{\bb R}^+$ are continuous functions with  $s^-(\varepsilon)\rightarrow1$ and $s^{+}(\varepsilon)\rightarrow1$ as $\varepsilon \rightarrow0$.
	Moreover, $s^-(\varepsilon)<1$ and $s^+(\varepsilon)>1$.
\end{lemma}
\begin{proof}
	For any $L'$ in $U_{\varepsilon}$,  we write it as $L'=L+  \varepsilon H$
for some $\bb R$-divisor $H$ in $N^1(X)$.
For any $s>0$, we can write
	\begin{equation*}
		L+\varepsilon H
		=\frac{1}{1+s}\left(L+s(L+\frac{(1+s)\varepsilon}{s}H)\right),
	\end{equation*}
	and set
	\begin{equation*}
		L_s:=L+\frac{(1+s)\varepsilon}{s}H.
	\end{equation*}
	Then by choosing $s$ small enough (determined later), which depends on $\varepsilon$, we can assume that both $(1+s)L-L_s$ and $L_s-(1-s)L$ are big.
	Indeed,
	\begin{equation*}
		(1+s)L-L_s
		=s\left(L-\frac{(1+s)\varepsilon}{s^2}H \right),
	\end{equation*}
	for instance, take $s=\varepsilon^{1/4}$, then $(1+s)L-L_s$ is big when $\varepsilon$ is small.
	
	Thus we have
\begin{flalign}
	S_{L'}(F)
	=&\int_0^{+\infty}\Vol(L'-xF)dx        \nonumber                    \\
	=&(1+s)^{-n}\int_0^{+\infty}
	\Vol(L+sL_s-(1+s)xF)dx   \nonumber             \\
	\geq&(1+s)^{-n}\int_0^{+\infty}
	\Vol(L+(s-s^2)L-(1+s)xF)dx       \nonumber       \\
	=&\left(\frac{1+s-s^2}{1+s}\right)^n
	\int_0^{+\infty}\Vol(L-\frac{1+s}{1+s-s^2}xF)dx    \nonumber      \\
	=&\left(\frac{1+s-s^2}{1+s}\right)^{n+1}  S_L(F).  \label{estimate of S-invariant}
\end{flalign}
On the other hand, similarly, we have
\begin{equation}
\label{estimate of S another direction}
	S_{L'}(F)\leq
	\left(\frac{1+s+s^2}{1+s}\right)^{n+1}S_L(F).
\end{equation}
By taking 
\begin{equation*}
	s^-(\varepsilon)
	=\left(1-\frac{\varepsilon^{1/2}}{1+\varepsilon^{1/4}}\right)^{n+1}
	\text{ and } 
	s^+(\varepsilon)
	=\left(1+\frac{\varepsilon^{1/2}}{1+\varepsilon^{1/4}}\right)^{n+1},
\end{equation*}
we finish the proof of Lemma \ref{S estimate lemma}.

\end{proof}

In what follows, we aim to establish the inequality (\ref{diff of beta}).
In fact, we do not need to show the inequality (\ref{diff of beta}) for any prime divisor $F$ over $X$.
By the definition of uniformly valuative stability, 
we introduce a subset of prime divisors over $X$ as in the next definition, on which it is clearly sufficient to test uniformly valuative stability.
\begin{definition}
	For any $L$ in ${\rm Amp}(X)$, let
	\begin{equation*}
		{\cal D}^{\rm ud}_L
		:=\{
		F\in {\rm PDiv}_{/X}: \beta_L(F)\leq C_LS_L(F)
		\},
	\end{equation*}
	for some constant $C_L>0$ (determined later).
\end{definition}  
It follows that we only need to prove the inequality (\ref{diff of beta}) for any $F$ in ${\cal D}^{\rm ud}_{L'}$. 
Since when $F\notin{\cal D}^{\rm ud}_{L'}$, it automatically satisfies the condition of uniformly valuative stability.

Then for any $F$ in ${\cal D}^{\rm ud}_{L'}$, we have 
\begin{flalign}
	&A_X(F)\Vol(L')+\int_0^{+\infty}n\langle(L'-xF)^{n-1}\rangle\cdot(K_X+\tilde{s}(L')L')dx  \nonumber \\
	\leq&\left(C_{L'}-n\mu(L')+(n+1)\tilde{s}(L')\right)S_{L'}(F),   
	\label{upper bound of A+R}
\end{flalign}
where we have used the Lemma \ref{integration by part lemma}. Now we choose $C_{L'}>0$ such that $C_{L'}-n\mu(L')+(n+1)\tilde{s}(L')\geq0$.
\begin{theorem}
\label{lower bdd theorem}
	Given a divisor $L$ in ${\rm Amp}(X)$, 
	there exists a constant $\varepsilon_0>0$ and a continuous function $f:{\bb R}^+\rightarrow{\bb R}^+$ with $\lim_{\varepsilon\rightarrow0}f(\varepsilon)=0$, such that for any $0<\varepsilon\leq\varepsilon_0$ and any $L'\in U_\varepsilon$, the inequality
	\begin{equation*}
	\beta_{L'}(F)-\beta_L(F)\geq
		- f(\varepsilon)S_{L'}(F)
	\end{equation*}
	is satisfied for all $F\in{\cal D}^{\rm ud}_{L'}$. 
Moreover, the choice of $f$ only depends on $X$ and $L$.
\end{theorem}

We first show the estimate of the second term of $\beta$-invariant, i.e. $\mu S$.
\begin{lemma}
\label{estimate of slope with S} 
For any $L$ in ${\rm Amp}(X)$, 
there exists a constant $\varepsilon_0>0$ and a continuous function $h:{\bb R}^+\rightarrow{\bb R}^+$ with $\lim_{\varepsilon\rightarrow0}h(\varepsilon)=0$, such that for any $0<\varepsilon\leq\varepsilon_0$ and any $L'\in U_\varepsilon$, the inequality
	\begin{equation}
	\label{estimate of slope}
	n\mu(L')S_{L'}(F)-n\mu(L)S_L(F)
	\geq-h(\varepsilon)nS_{L'}(F)  
	\end{equation}
	is satisfied for all $F\in{\cal D}^{\rm ud}_{L'}$. 
Moreover, the choice of $h$ only depends on $X$ and $L$.
\end{lemma}
\begin{proof}
For simplicity, we denote 
$$
s_-(n):=\left(1-\frac{s^2}{1+s} \right)^n<1 \text{ and }\ 
s_+(n):=\left(1+\frac{s^2}{1+s} \right)^n>1.
$$
For any $L'$ in $U_{\varepsilon}$, we can write
	\begin{equation*}
		L'=L+\varepsilon H
		=\frac{1}{1+s}\left(L+sL_s\right)
	\end{equation*}
	in the same way as the proof of Lemma \ref{S estimate lemma}, for some $\bb R$-divisor $H$ and  $L_s$ in $N^1(X)$. 
Thus we have
\begin{align*}
	\Vol(L')
	=&(1+s)^{-n}\Vol(L+sL_s)               \\
	\geq&(1+s)^{-n}\Vol(L+(s-s^2)L) \\         
	=&s_{-}(n)\Vol(L).
\end{align*}
Similarly, one obtains
\begin{align*}
	\Vol(L')
	\leq&s_{+}(n)\Vol(L).
\end{align*}
The proof falls naturally into two cases.

\begin{itemize}
	\item[(1)] When $\mu(L')\geq0$, then we compute
\begin{flalign}
	&n\mu(L')S_{L'}(F)-n\mu(L)S_L(F)    \nonumber     \\
	\geq&nS_L(F)(s^-(\varepsilon)\mu(L')-\mu(L))      \nonumber   \\
	=&nS_L(F)\left((s^-(\varepsilon)-1)\mu(L')+\frac{-K_X\cdot(L')^{n-1}}{\Vol(L')}-\frac{-K_X\cdot L^{n-1}}{\Vol(L)}\right)   \nonumber   \\
	\geq&nS_L(F)\left((s^-(\varepsilon)-1)\mu(L')+\frac{-K_X\cdot(L')^{n-1}}{s_+(n)\Vol(L)}-\frac{-K_X\cdot L^{n-1}}{\Vol(L)} \right)   \nonumber  \\
	\geq&nS_L(F)\left((s^-(\varepsilon)-1)\mu(L')+(s_+(n)^{-1}-1)\frac{-K_X\cdot(L')^{n-1}}{\Vol(L)}\right.                    \nonumber          \\
	&\left.+\frac{1}{\Vol(L)}(-K_X\cdot(L')^{n-1}-(-K_X)\cdot L^{n-1}) \right)         \nonumber         \\
	\geq&nS_L(F)\left((s^-(\varepsilon)-1)\mu(L')+(\frac{1}{s_+(n)}-1)s_+(n)\mu(L')\right.                              \nonumber \\
	&\left.+\frac{1}{\Vol(L)}((-K_X)\cdot\varepsilon H ((L')^{n-2}+(L')^{n-3}\cdot L+\cdots + L^{n-2}) \right)             \nonumber     \\
	\geq&nS_L(F)\Big((s^-(\varepsilon)-s_+(n))\mu(L') \nonumber \\                             
	&\left.+\varepsilon\frac{1}{\Vol(L)}((-K_X)\cdot H ((L')^{n-2}+(L')^{n-3}\cdot L+\cdots + L^{n-2}) \right).          \nonumber      
\end{flalign}
In general, we do not know the sign of 
\begin{equation*}
	\frac{1}{\Vol(L)}((-K_X)\cdot H ((L')^{n-2}+(L')^{n-3}\cdot L+\cdots + L^{n-2}).
\end{equation*}
But we can cancel it directly if it is nonnegative.
Therefore, without loss of generality, we may assume that it is negative. Then
\begin{flalign}
	n\mu(L')S_{L'}(F)-n\mu(L)S_L(F)
	\geq&(-h_1(\varepsilon)-g(\varepsilon))nS_L(F)           \nonumber \\
	\geq&(-h_1(\varepsilon)-g(\varepsilon))s^-(\varepsilon)^{-1}nS_{L'}(F),     
\end{flalign}
where 
\begin{equation}
	g(\varepsilon)
	=-\varepsilon \frac{1}{\Vol(L)}((-K_X)\cdot H ((L')^{n-2}+(L')^{n-3}\cdot L+\cdots + L^{n-2})
\end{equation}
which is a polynomial in $\varepsilon$ with degree $n-1$ and $g(0)=0$, whose coefficients depend on $-K_X$, $L$, $H$, and the leading term is $\Vol(L)^{-1}(-K_X)\cdot H^{n-1}$, and $h_1:{\bb R}^+\rightarrow{\bb R}^+$ is a continuous function with  $h_1(\varepsilon)\rightarrow0$ as $\varepsilon\rightarrow0$, which depends on $\mu(L')$.

In fact, $h_1$ is independent of the choice of $L'$.
Since $g$ is a polynomial with degree $n-1$ and $L'$ can be represented by a basis of Nef cone (see the following Lemma \ref{a and varepsilon}), then the choice of $g$ only depends on $X$ and $L$.

\item[(2)] When $\mu(L')\leq0$, the computation is similar. We omit it.
\end{itemize}
This completes the proof of Lemma \ref{estimate of slope with S} by taking $h=(h_1+g)s^-(\varepsilon)^{-1}$.
\end{proof}

\begin{lemma}
\label{a and varepsilon}
	There exists a constant $a>0$, which depends on $\varepsilon$ and $L$, such that 
	\begin{equation*}
		(1-a)L\leq L'\leq(1+a)L
	\end{equation*}
	for any $L'$ in $U_{\varepsilon}$.
	Moreover, such $a$ can be chosen as small as we wish by choosing $\varepsilon$ small.
\end{lemma}
\begin{proof}
	For any $L'$ in $U_{\varepsilon}$,  we write it as $L'=L+ H$
for some $\bb R$-divisor $H$ in $N^1(X)$ with $\|H\|<\varepsilon$.
Set $\rho:=\dim_{\bb R}N^1(X)$. 
Since $L$ is ample, there exists a basis $(A_1,\ldots,A_{\rho})$ of $N^1(X)$ with each $A_i$ in ${\rm Nef}(X)$, and there exists some $t_1,\ldots,t_{\rho}\in{\bb R}_{>0}$ such that
$L=\sum_{i=1}^\rho t_iA_i$ with $\sum_{i=1}^\rho t_i=1$.
Set $t_0=\min_it_i\in{\bb R}_{>0}$.
We may assume that the norm $\|\cdot\|$ is given by
\begin{equation*}
	\left\|\sum_{i=1}^\rho s_iA_i \right\|
	:=\sum_{i=1}^{\rho}\vert s_i\vert.
\end{equation*}
Set $H=\sum_{i=1}^\rho r_iA_i$ with $\|H\|<\varepsilon$ (i.e. $\sum_{i=1}^\rho \vert r_i\vert<\varepsilon$). 
Then we have
\begin{align*}
	L'=L+H=\sum_{i=1}^\rho(t_i+r_i)A_i
	<\sum_{i=1}^\rho(t_i+\varepsilon)A_i
	=\sum_{i=1}^\rho(t_i+\frac{\varepsilon}{t_0}t_0)A_i
	\leq(1+\frac{\varepsilon}{t_0})\sum_{i=1}^\rho t_iA_i,
\end{align*}  
also
\begin{align*}
	L'=L+H=\sum_{i=1}^\rho(t_i+r_i)A_i
	>\sum_{i=1}^\rho(t_i-\varepsilon)A_i
	=\sum_{i=1}^\rho(t_i-\frac{\varepsilon}{t_0}t_0)A_i
	\geq(1-\frac{\varepsilon}{t_0})\sum_{i=1}^\rho t_iA_i.
\end{align*}  
The proof is completed by taking $a=\varepsilon/t_0$, where $t_0=\min_i t_i$.
\end{proof}
\begin{remark}
       Consistent with the notation in Section \ref{Positive intersection product}, $"\leq"$ means that their difference is a psef class.
       In fact, $(1+a)L-L'$ and $L'-(1-a)L$ are nef according to the proof of Lemma \ref{a and varepsilon}.
\end{remark}

We now turn to the proof of Theorem \textup{\ref{lower bdd theorem}}.

\begin{proof}
	For any $L'$ in $U_{\varepsilon}$, we can write
	\begin{equation*}
		L'=L+\varepsilon H
		=\frac{1}{1+s}\left(L+sL_s\right)
	\end{equation*}
	in the same way as the proof of Lemma \ref{S estimate lemma}, for some $\bb R$-divisor $H$ and  $L_s$ in $N^1(X)$ such that $(1+s)L-L_s$ and $L_s-(1-s)L$ are big when $\varepsilon$ is small enough, where $s=\varepsilon^{1/4}$.

~\\
We divide into following these two cases, 
\begin{enumerate}[(1)]
	\item\label{case 1} 

One assume $\mu(L')\geq0$, then $\tilde{s}(L')\geq0$.
\begin{flalign}
	&\beta_{L'}(F)-\beta_L(F)    \nonumber        \\
	=&A_X(F)(\Vol(L')-\Vol(L))                          
	+n\mu(L')S_{L'}(F)-n\mu(L)S_L(F)    \nonumber \\
	 &+\int_0^{+\infty}n\langle(L'-xF)^{n-1}\rangle\cdot (K_X+\tilde{s}(L')L')dx      
	 -\tilde{s}(L')\int_0^{+\infty}n\langle(L'-xF)^{n-1}\rangle{\cdot}L'dx                \nonumber \\
	 &-\int_0^{+\infty}n\langle(L-xF)^{n-1}\rangle{\cdot}(K_X+\tilde{s}(L')L')dx         \nonumber         \\
	 &+\int_0^{+\infty}n\langle(L-xF)^{n-1}\rangle{\cdot}(K_X+\tilde{s}(L')L')dx        \nonumber           \\
	 &-\int_0^{+\infty}n\langle(L-xF)^{n-1}\rangle\cdot K_X dx  \nonumber  \\
	\geq&n\mu(L')S_{L'}(F)-n\mu(L)S_L(F)
	+A_X(F)\Vol(L')(1-s_-(n)^{-1})    \nonumber  \\
	&-\tilde{s}(L')\int_0^{+\infty}n\langle(L'-xF)^{n-1}\rangle\cdot L'dx 
	+\tilde{s}(L')\int_0^{+\infty}n\langle(L-xF)^{n-1}\rangle\cdot L' dx     \nonumber    \\
	&+\int_0^{+\infty}n\left(\langle(L'-xF)^{n-1}\rangle-\langle(L-xF)^{n-1}\rangle\right)
	    {\cdot}(K_X+\tilde{s}(L')L')dx.  \label{first case of diff of beta}
\end{flalign}
By Lemma \ref{a and varepsilon}, we can take a small positive constant $a$ (recall $a=\varepsilon/t_0$) such that
\begin{equation*}
	(1-a)L\leq L'\leq (1+a)L,
\end{equation*}
for any $L'$ in $U_{\varepsilon}$.
Then one obtains
\begin{equation*}
	(1-a)L-xF\leq L'-xF\leq (1+a)L-xF.
\end{equation*}
Therefore,
by the continuity and homogeneity of the positive intersection product (see Proposition \ref{cont of positive intersection product} or \cite[Proposition 2.9]{BFJ09}),  
we have
\begin{equation}
\label{cont of pos inters prod}
	(1-a)^{n-1}\langle(L-\frac{x}{1-a} F)^{n-1}\rangle
	\leq\langle(L'-xF)^{n-1}\rangle
	\leq(1+a)^{n-1}\langle(L-\frac{x}{1+a} F)^{n-1}\rangle.
\end{equation}
Since $K_X+\tilde{s}(L')L'$ is nef, we have
\begin{align*}
	&(1-a)^{n-1}\langle(L-\frac{x}{1-a} F)^{n-1}\rangle\cdot(K_X+\tilde{s}(L')L') \\
	\leq&\langle(L'-xF)^{n-1}\rangle\cdot(K_X+\tilde{s}(L')L')            \\
	\leq&(1+a)^{n-1}\langle(L-\frac{x}{1+a} F)^{n-1}\rangle\cdot(K_X+\tilde{s}(L')L').
\end{align*}
It follows that
\begin{align*}
    &\int_0^{+\infty}n\langle(L'-x F)^{n-1}\rangle\cdot(K_X+\tilde{s}(L')L')dx      \\
	\geq&(1-a)^{n-1}\int_0^{+\infty}n\langle(L-\frac{x}{1-a}F)^{n-1}\rangle
	    {\cdot}(K_X+\tilde{s}(L')L')dx     \\
	=&(1-a)^n\int_0^{+\infty}n\langle(L-xF)^{n-1}\rangle\cdot(K_X+\tilde{s}(L')L')dx.
\end{align*}
Thus, we obtain
\begin{align*}
	&\int_0^{+\infty}n\left(\langle(L'-xF)^{n-1}\rangle-\langle(L-xF)^{n-1}\rangle\right)
	    {\cdot}(K_X+\tilde{s}(L')L')dx      \\
	\geq&(1-(1-a)^{-n})\int_0^{+\infty}n\langle(L'-xF)^{n-1}\rangle\cdot(K_X+\tilde{s}(L')L')dx.
\end{align*}
Recall $s=\varepsilon^{1/4}$, when we choose $\varepsilon$ small enough, then $a$ ($=\varepsilon/t_0$, see Lemma \ref{a and varepsilon}) can be chosen small enough, such that
\begin{equation*}
	1-s_-(n)^{-1}\leq 1-(1-a)^{-n}.
\end{equation*}
Then, we obtain
\begin{flalign}
	&A_X(F)\Vol(L')(1-s_-(n)^{-1})  \nonumber    \\
	+&\int_0^{+\infty}n\left(\langle(L'-xF)^{n-1}\rangle-\langle(L-xF)^{n-1}\rangle\right)
	    {\cdot}(K_X+\tilde{s}(L')L')dx   \nonumber   \\
	\geq&(1-s_-(n)^{-1})\left(
	A_X(F)\Vol(L')+\int_0^{+\infty}n\langle(L'-xF)^{n-1}\rangle\cdot(K_X+\tilde{s}(L')L')dx
	\right)       \nonumber        \\
	\geq&(1-s_-(n)^{-1})\left(C_{L'}-n\mu(L')+(n+1)\tilde{s}(L')\right)S_{L'}(F) 
	\label{lower bdd of A plus pos inters term}
\end{flalign}
where we used 
(\ref{upper bound of A+R}) and Lemma \ref{integration by part lemma} for the second inequality.

Since $L'$ is ample, by (\ref{cont of pos inters prod}), we have
\begin{align*}
	(1-a)^{n-1}\langle(L-\frac{x}{1-a} F)^{n-1}\rangle\cdot L' 
	\leq&\langle(L'-xF)^{n-1}\rangle\cdot L'            \\
	\leq&(1+a)^{n-1}\langle(L-\frac{x}{1+a} F)^{n-1}\rangle\cdot L'.
\end{align*}
Then one obtains
\begin{align*}
	\int_0^{+\infty}n\langle(L'-xF)^{n-1}\rangle\cdot L'dx 
	\leq&\int_0^{+\infty}(1+a)^{n-1}n\langle(L-\frac{x}{1+a} F)^{n-1}\rangle\cdot L' dx \\
	=&(1+a)^n\int_0^{+\infty}n\langle(L-xF)^{n-1}\rangle\cdot L' dx.
\end{align*}
It follows that
\begin{flalign}
	&\tilde{s}(L')\int_0^{+\infty}n\langle(L-xF)^{n-1}\rangle\cdot L' dx  
	-\tilde{s}(L')\int_0^{+\infty}n\langle(L'-xF)^{n-1}\rangle\cdot L'dx     \nonumber     \\  
	\geq&\tilde{s}(L')((1+a)^{-n}-1)\int_0^{+\infty}n\langle(L'-xF)^{n-1}\rangle\cdot L' dx   \nonumber \\
	=&\tilde{s}(L')((1+a)^{-n}-1)(n+1)S_{L'}(F). \label{lower bdd of pos inters with L'}
\end{flalign}
Note that here we have used $\tilde{s}(L')\geq0$ and Lemma \ref{integration by part lemma}.

Combining (\ref{first case of diff of beta}), (\ref{lower bdd of A plus pos inters term}), (\ref{lower bdd of pos inters with L'}), and (\ref{estimate of slope}) , we have 
\begin{align*}
	&\beta_{L'}(F)-\beta_L(F)    \\
	\geq& -h(\varepsilon) nS_{L'}(F)      
	+\Big(
	(1-s_-(n)^{-1})\big(C_{L'}-n\mu(L')  
	+(n+1)\tilde{s}(L')\big) \\
	&+(n+1)\tilde{s}(L')((1+a)^{-n}-1)
	\Big)S_{L'}(F)                      \\
	\geq&-f(\varepsilon)S_{L'}(F),                         
\end{align*}
where $f:{\bb R}^+\rightarrow{\bb R}^+$ is a continuous function with  $f(\varepsilon)\rightarrow0$ as $\varepsilon\rightarrow0$, which depends on $\mu(L')$, $\tilde{s}(L')$ and $C_{L'}$ and intersection numbers $(-K_X)\cdot (L')^{k}\cdot L^{n-1-k}$ for $k=0,\ldots,n-1$.

By definition, we know that $\mu(L')$ and $\tilde{s}(L')$ are continuous with respect to $L'$. 
Thus, we can choose $C_{L'}$ continuously depending on $L'$.
Therefore, the choice of $f$ only depends on $X$ and $L$.

\item One assumes $\mu(L')\leq0$, then $s(L')\leq0$. We use the same idea of case (\ref{case 1}).
\begin{flalign}
	&\beta_{L'}(F)-\beta_L(F)      \nonumber       \\
	=&A_X(F)(\Vol(L')-\Vol(L))                          
	+n\mu(L')S_{L'}(F)-n\mu(L)S_L(F)   \nonumber
	  \\
	 &+\int_0^{+\infty}n\langle(L'-xF)^{n-1}\rangle\cdot K_Xdx      
	 +\int_0^{+\infty}n\langle(L'-xF)^{n-1}\rangle{\cdot}(-K_X-s(L')L')dx                 \nonumber \\
	 &-\int_0^{+\infty}n\langle(L'-xF)^{n-1}\rangle{\cdot}(-K_X-s(L')L')dx                      \nonumber \\
	 &+\int_0^{+\infty}n\langle(L-xF)^{n-1}\rangle{\cdot}(-K_X-s(L')L')dx                   
	 +s(L')\int_0^{+\infty}n\langle(L-xF)^{n-1}\rangle\cdot L' dx   \nonumber \\
	\geq&n\mu(L')S_{L'}(F)-n\mu(L)S_L(F)
	+A_X(F)\Vol(L')(1-s_-(n)^{-1})     \nonumber \\
	&-s(L')\int_0^{+\infty}n\Big(\langle(L'-xF)^{n-1}\rangle-\langle(L-xF)^{n-1}\rangle\Big)\cdot L' dx    \nonumber     \\
	&+\int_0^{+\infty}n\left(\langle(L-xF)^{n-1}\rangle-\langle(L'-xF)^{n-1}\rangle\right)
	    {\cdot}(-K_X-s(L')L')dx.          \label{second case of diff of beta}
\end{flalign}
By (\ref{cont of pos inters prod}) and Lemma \ref{integration by part lemma}, we have
\begin{flalign}
	&\int_0^{+\infty}n\left(\langle(L-xF)^{n-1}\rangle-\langle(L'-xF)^{n-1}\rangle\right)
	    {\cdot}(-K_X-s(L')L')dx    \nonumber    \\
	\geq&((1+a)^{-n}-1)\int_0^{+\infty}n\langle(L'-xF)^{n-1}\rangle
	    {\cdot}(-K_X-s(L')L')dx     \nonumber   \\
	=&(1-(1+a)^{-n})\int_0^{+\infty}n\langle(L'-xF)^{n-1}\rangle\cdot K_X       \nonumber                   \\
	&+((1+a)^{-n}-1)(-s(L'))(n+1)S_{L'}(F).   \label{lower bdd of pos inters with K_X}
\end{flalign}
Since $F$ belongs to ${\cal D}^{\rm ud}_{L'}$, one can obtain
\begin{flalign}
	A_X(F)\Vol(L')(1-s_-(n)^{-1})
	\geq&(1-s_-(n)^{-1})(C_{L'}-n\mu(L'))S_{L'}(F)          \nonumber \\
	&-(1-s_-(n)^{-1})\int_0^{+\infty}n\langle(L'-xF)^{n-1}\rangle\cdot K_Xdx.   \label{lower bdd of A}
\end{flalign}
Since $L'$ is ample, by (\ref{cont of pos inters prod}) and Lemma \ref{integration by part lemma}, we obtain
\begin{flalign}
	&(-s(L'))\int_0^{+\infty}n\Big(\langle(L'-xF)^{n-1}\rangle-\langle(L-xF)^{n-1}\rangle\Big)\cdot L' dx     \nonumber    \\
	\geq&(-s(L'))(1-(1-a)^{-n})\int_0^{+\infty}n\langle(L'-xF)^{n-1}\rangle\cdot L'dx    \nonumber    \\
	=&(-s(L'))(1-(1-a)^{-n})(n+1)S_{L'}(F).     \label{lower of pos inters with L' 2}
\end{flalign}

In addition, we have the following natural lower bound,
\begin{flalign}
	&\int_0^{+\infty}n\langle(L'-xF)^{n-1}\rangle\cdot K_Xdx   \nonumber  \\
	=&\int_0^{+\infty}n\langle(L'-xF)^{n-1}\rangle\cdot(K_X+\tilde{s}(L')L')dx
	-\tilde{s}(L')\int_0^{+\infty}n\langle(L'-xF)^{n-1}\rangle\cdot L'dx                              \nonumber \\
	\geq&-\tilde{s}(L')(n+1)S_{L'}(F).   \label{natural lower bdd}
\end{flalign}
Combining (\ref{estimate of slope}) and (\ref{second case of diff of beta})-(\ref{natural lower bdd}), we have
\begin{align*}
	&\beta_{L'}(F)-\beta_L(F) \\
	\geq&-h(\varepsilon)nS_{L'}(F) 
	+\Big((1-s_-(n)^{-1})(C_{L'}-n\mu(L')) \\
	&+((1+a)^{-n}-1)(-s(L'))(n+1)+(-s(L'))(1-(1-a)^{-n})(n+1)
	\Big)S_{L'}(F)               \\
	&+\Big(1-(1+a)^{-n}+s_-(n)^{-1}-1  \Big)\int_0^{+\infty}n\langle(L'-xF)^{n-1}\rangle\cdot K_Xdx \\
	\geq&-h(\varepsilon)nS_{L'}(F)
	+\Big((1-s_-(n)^{-1})(C_{L'}-n\mu(L')) \\
	&+((1+a)^{-n}-(1-a)^{-n})(-s(L'))(n+1)
	\Big)S_{L'}(F)               \\
	&+\Big(s_-(n)^{-1}-(1+a)^{-n} \Big)
	(-\tilde{s}(L'))(n+1)S_{L'}(F)     \\
	\geq&	-f(\varepsilon)S_{L'}(F).
\end{align*}
where $f:{\bb R}^+\rightarrow{\bb R}^+$ is a continuous function with  $f(\varepsilon)\rightarrow0$ as $\varepsilon\rightarrow0$, which depends on $\mu(L')$, $\tilde{s}(L')$, $s(L')$ and $C_{L'}$ and intersection numbers $(-K_X)\cdot (L')^{k}\cdot L^{n-1-k}$ for $k=0,\ldots,n-1$.

Similar to case (1), we can choose a continuous function $f$ which only depends on $X$ and $L$.

\end{enumerate}
By combining  above these two cases, 
we complete the proof of Theorem \ref{lower bdd theorem}.
\end{proof}

Finally, we finish the proof of the main Theorem. 

\begin{proof}[Proof of Theorem \textup{\ref{UVs open}}] 
For any $L$ in ${\rm UVs}$, by Theorem \ref{lower bdd theorem},
there exists a constant $\varepsilon_0>0$ and a continuous function $f:{\bb R}^+\rightarrow{\bb R}^+$ with $\lim_{\varepsilon\rightarrow0}f(\varepsilon)=0$, 
which only depends $X$ and $L$, such that for any $0<\varepsilon\leq\varepsilon_0$ and any $L'\in U_\varepsilon$, the inequality
	\begin{equation*}
	\beta_{L'}(F)-\beta_L(F)\geq
		-f(\varepsilon)S_{L'}(F)
	\end{equation*}
	is satisfied for all $F$ in ${\cal D}^{\rm ud}_{L'}$.
	Since $L$ is uniformly valuatively stable, combining with (\ref{idea of uvs}), we have
	\begin{equation*}
		\beta_{L'}(F)\geq
		\left(\frac{\varepsilon_L}{2}-f(\varepsilon) \right)S_{L'}(F),
	\end{equation*}
	for all $F$ in ${\cal D}^{\rm ud}_{L'}$.
	It follows that there exists an $\varepsilon_0>0$ such that 
	\begin{equation*}
		\frac{\varepsilon_L}{2}-f(\varepsilon)
		>0
	\end{equation*}
	for all $L$ in $U_{\varepsilon}$ and any $0<\varepsilon\leq\varepsilon_0$,  
	Then for any $L$ in $U_{\varepsilon}$, we have
	\begin{equation*}
		\beta_{L'}(F)\geq
		\varepsilon_{L'} S_{L'}(F)
	\end{equation*}
	for some constant $\varepsilon_{L'}>0$ and all $F$ in ${\cal D}^{\rm ud}_{L'}$.
	Thus, $L'$ belongs to ${\rm UVs}$ for any $L'$ in $U_{\varepsilon}$.
	
	Finally, by definitions of $\beta$ and $S$-invariant, we have
	\begin{equation*}
		\beta_{kL'}(F)=k^n\beta_{L'}(F), \text{ and }
		S_{kL'}(F)=k^{n+1}S_{L'}(F),
	\end{equation*}
	for $k>0$. 
	Then ${\bb R}_{+}U_{\varepsilon}\subset {\rm UVs}$.
	Therefore, the uniformly valuative stability locus $\rm UVs$ is an open subcone of ${\rm Amp}(X)$.
\end{proof}

\section{Continuity of the uniformly valuative stability threshold}
\label{Continuity of uniformly valuative stability threshold}
As an immediate application of Theorem \ref{lower bdd theorem} and \ref{UVs open}, in this section, we show the continuity of the uniformly valuative stability threshold.

\begin{definition}
\label{u.v.s.t.}
For any $L$ in ${\rm Amp}(X)$,
the \textit{uniformly valuative stability threshold} of $L$ is defined to be
\begin{equation*}
	\zeta(L):=\sup\{x\in{\bb R}|\beta_L(F)\geq xS_L(F) \text{ for any }F\in{\rm PDiv}_{/X} \}.
\end{equation*}
\end{definition}
In fact, when $(X,L)=(X,-K_X)$ is Fano, we have $\zeta(L)=\delta(X)-1$.
This is the main motivation to study the $\zeta$-invariant.

Recall the definition of $\delta$-invariant, due to Blum and Jonsson \cite{BJ20},
\begin{equation*}
	\delta(L)=\inf_{F\in{\rm PDiv}_{/X}} \frac{A_X(F)\Vol(L)}{S_L(F)},
\end{equation*}
see \cite{FO18} for the original definition of $\delta$-invariant.
Thus, one obtains
\begin{equation*}
	A_X(F)\Vol(L)\geq \delta(L)S_L(F)
\end{equation*}
for any $F$ in ${\rm PDiv}_{/X}$.
By (\ref{natural lower bdd}), we have a natural lower bound
\begin{equation*}
	\beta_L(F)\geq(\delta(L)+n\mu(L)-(n+1)\tilde{s}(L))S_L(F),
\end{equation*}
i.e.
\begin{equation*}
	\zeta(L)\geq\delta(L)+n\mu(L)-(n+1)\tilde{s}(L).
\end{equation*}
One can take a $c_{L}>0$ such that $\delta(L)+n\mu(L)-(n+1)\tilde{s}(L)+c_L>0$. 
Thus now we set $C_L:=\delta(L)+n\mu(L)-(n+1)\tilde{s}(L)+c_L>0$. in the definition of ${\cal D}^{\rm ud}_{L}$.
We also define
\begin{equation*}
	\zeta^{\rm ud}(L)
	:=\sup\{x\in{\bb R}|\beta_L(F)\geq xS_L(F) \text{ for any }F\in{\cal D}^{\rm ud}_{L} \}.
\end{equation*}
By definition, one obtains
\begin{equation*}
	C_L\geq \zeta^{\rm ud}(L).
\end{equation*}
\begin{lemma}
	For any $L$ in ${\rm Amp}(X)$, we have
	\begin{equation*}
		\zeta(L)=\zeta^{\rm ud}(L).
	\end{equation*}
\end{lemma}
\begin{proof}
	By definitions of $\zeta^{\rm ud}(L)$ and $\zeta(L)$, we have
	\begin{equation*}
		\zeta(L)\leq \zeta^{\rm ud}(L).
	\end{equation*}
For any $F\notin{\cal D}^{\rm ud}_L$, then
\begin{equation*}
	\beta_L(F)\geq C_LS_L(F)
	\geq \zeta^{\rm ud}(L)S_L(F).
\end{equation*}
Thus, for any $F$ in ${\rm PDiv}_{/X}$, we have
\begin{equation*}
	\beta_L(F)\geq \zeta^{\rm ud}(L)S_L(F),
\end{equation*}
i.e.
\begin{equation*}
	\zeta(L)\geq \zeta^{\rm ud}(L).
\end{equation*}
\end{proof}
\begin{theorem}
\label{cont of u.v.s.t.}
	The uniformly valuative stability threshold
	\begin{equation*}
		{\rm Amp}(X)\ni L\mapsto \zeta(L)\in{\bb R}
	\end{equation*}
	is continuous on the ample cone.
\end{theorem}
\begin{proof}
	For any $L$ in ${\rm Amp(X)}$ and any $\varepsilon>0$, we aim to show that there exists a small open neighbourhood $U_{\theta}$ of $L$ in ${\rm Amp}(X)$ such that for any $L'$ in $U_{\theta}$ satisfying 
	\begin{equation*}
		|\zeta(L')-\zeta(L)|<\varepsilon.
	\end{equation*}
	By Theorem \ref{lower bdd theorem},
	for any $L'$ in $U_{\theta}$ satisfies the following inequality 
	\begin{equation*}
	\beta_{L'}(F)-\beta_L(F)\geq
		-f(\theta)S_{L'}(F)
	\end{equation*}
	for any $F$ in ${\cal D}^{\rm ud}_{L'}$, 
	where $f$ is a continuous function with  $f(\theta)\rightarrow0$ as $\theta\rightarrow0$. 
	Moreover, $f$ only depends on $X$ and $L$.
	Thus, we have
	\begin{align*}
		\beta_{L'}(F)
		\geq&\zeta(L)S_L(F)-f(\theta)S_{L'}(F) \\
		=&(\zeta(L)+c_L)S_L(F)-c_LS_L(F)-f(\theta)S_{L'}(F)          \\
		\geq&\left((\zeta(L)+c_L)s^-(\theta)
		-c_Ls^+(\theta)-f(\theta)\right)S_{L'}(F)   \\
		=&\left(
		\zeta(L)-(1-s^-(\theta))\zeta(L)-c_L(s^+(\theta)-s^-(\theta))-f(\theta)
		\right)S_{L'}(F)
	\end{align*}
	for any $F$ in ${\cal D}^{\rm ud}_{L'}$.
	Thus one obtains
	\begin{equation*}
		\zeta(L')=\zeta^{\rm ud}(L')
		\geq\zeta(L)-\left(
		(1-s^-(\theta))\zeta(L)+c_L(s^+(\theta)-s^-(\theta))+f(\theta)\right).
	\end{equation*}
	We can take a small enough constant $\theta>0$ such that
	\begin{equation*}
		(1-s^-(\theta))\zeta(L)+c_L(s^+(\theta)-s^-(\theta))+f(\theta)
		<\varepsilon.
	\end{equation*}
	Thus, we have
	\begin{equation}
	\label{lower s.c.}
		\zeta(L')-\zeta(L)>-\varepsilon.
	\end{equation}
	On the other hand, by replacing $L$ by $L'$ and write $L=L'-\theta H$ in Theorem \ref{lower bdd theorem}, we have
	\begin{equation*}
	\beta_{L}(F)-\beta_{L'}(F)\geq
		-f(\theta)S_{L}(F)
	\end{equation*}
	for any $F$ in ${\cal D}^{\rm ud}_{L}$, 
	where $f$ is a continuous function with  $f(\theta)\rightarrow0$ as $\theta\rightarrow0$. 
	Moreover, $f$ only depends on $X$ and $L$.
	Similarly, we can compute
	\begin{align*}
		\beta_L(F)
		\geq\left( \zeta(L')s^+(\theta)^{-1}-c_{L'}(s^-(\theta)^{-1}-s^+(\theta)^{-1})-f(\theta)
		\right)S_{L}(F),
	\end{align*}
	for any $F$ in ${\cal D}^{\rm ud}_{L}$.
	One obtains
	\begin{equation*}
		\zeta(L)=\zeta^{\rm ud}(L)
		\geq \zeta(L')s^+(\theta)^{-1}-c_{L'}(s^-(\theta)^{-1}-s^+(\theta)^{-1})-f(\theta).
	\end{equation*}
	Then, we have 
	\begin{equation*}
	\zeta(L')\leq\zeta(L)+
		(s^{+}(\theta)-1)\zeta(L) 
		+c_{L'}(s^+(\theta)s^-(\theta)^{-1}-1)+s^+(\theta)f(\theta).
	\end{equation*}
	One can choose a $c_{L'}$ depending on $L'$ continuously  since $\delta(\cdot)$, $\mu(\cdot)$ and $\tilde{s}(\cdot)$ are continuous on ${\rm Amp}(X)$.
	Then we take $\theta>0$ small enough such that
	\begin{equation*}
		(s^{+}(\theta)-1)\zeta(L) 
		+c_{L'}(s^+(\theta)s^-(\theta)^{-1}-1) +s^+(\theta)f(\theta)
		<\varepsilon.
	\end{equation*}
	Thus, we have
	\begin{equation*}
		\zeta(L')-\zeta(L)<\varepsilon.
	\end{equation*}
	Together with (\ref{lower s.c.}), we finish the proof of Theorem \ref{cont of u.v.s.t.}.
	\end{proof}

\section{Valuative stability for transcendental classes}
\label{Valuative stability for transcendental classes}
In this section, let $X$ be a projective manifold. We extend the valuative stability to the K\"ahler cone of projective manifolds. 

Denote by $\cal K$ the \textit{K\"ahler cone} of $X$ and $\cal E$ the \textit{pseudo-effective cone} in $H^{1,1}(X,\R)$. 
The interior ${\cal E}^{\circ}$ of the psef cone is an open subcone, whose element is called \textit{big}.

Recall the definition of the volume of a big class $\alpha$ in ${\cal E}^{\circ}$ (\cite[Definition 3.2]{BDPP12}), 
\begin{equation*}
	\Vol(\alpha):=\sup_{T\in\alpha}\int_{\widetilde{X}}\gamma^n>0,
\end{equation*}
where the supremum is taken over all K\"ahler currents $T\in\alpha$ with logarithmic poles, and $\pi^*T=[E]+\gamma$ with respect to some modification $\pi:\widetilde{X}\rightarrow X$ for an effective $\bb Q$-divisor $E$ and a closed semi-positive form $\gamma$
(or see \cite[Definition 1.17]{BEGZ10} for a definition in the sense of the pluripotential theory).

Let $\alpha\in{\cal K}$ be a K\"ahler class of $X$, for any prime divisor $F$ over $X$, then $\Vol(\alpha-x[F])$ is well-defined for some small $x>0$. 
Since $\pi^*\alpha$ may not be K\"ahler on $Y$, but it is still big.
Therefore, by the openness of the big cone ${\cal E}^{\circ}$, we can define
\begin{equation}
	\tau_{\alpha}(F):=\{x\in\R|\Vol(\alpha-x[F])>0 \}.
\end{equation} 
It follows that the $S$-invariant is well-defined, denoted by $S_{\alpha}(\cdot)$.
Similarly, for any K\"ahler class $\alpha$, we also define
\begin{equation*}
	\mu(\alpha):=\frac{c_1(X)\cdot\alpha^{n-1}}{\alpha^n},      
\end{equation*}
\begin{equation*}
	s(\alpha):=\sup\{s\in\mathbb{R}|c_1(X)-s\alpha \text{ is K\"ahler }\},    
\end{equation*}
\begin{equation*}
	\tilde{s}(\alpha):=\inf\{s\in\mathbb{R}|-c_1(X)+s\alpha \text{ is K\"ahler }\}.
\end{equation*}
We  have $s(\alpha)\leq\mu(\alpha)\leq\tilde{s}(\alpha)$.

In \cite{BDPP12}, the authors established the perfect theory of the positive intersection product of big classes on  compact K\"ahler manifolds.
\begin{theorem}[{\textup{\cite[Theorem 3.5]{BDPP12}}}]
\label{thm 3.5 of BDPP}
	 Let $X$ be a compact K\"ahler manifold.
We denote here by $H^{k,k}_{\geq0}(X)$ the cone of cohomology classes of type $(k,k)$ which have non-negative intersection with all closed semi-positive smooth forms of bidegree $(n-k,n-k)$.
\begin{enumerate}[$(\rm i)$]
	\item For each integer $k = 1,2,\ldots,n$, there exists a canonical "movable intersection product"
	    \begin{equation*}
	    	{\cal E}\times\cdots\times{\cal E}\rightarrow H^{k,k}_{\geq0}(X),\quad
	    	(\alpha_1,\ldots,\alpha_k)\mapsto \langle\alpha_1\cdot\alpha_2\cdots\alpha_k\rangle
	    \end{equation*}
	    such that $\Vol(\alpha)=\langle\alpha^n\rangle$ whenever $\alpha$ is a big class $($see Remark \ref{analytic pos inters}$)$.
	 \item \label{superadditive}
	 The product is increasing, homogeneous of degree $1$ and super-additive in each argument, i.e.
	    \begin{equation*}
	    	\langle\alpha_1\cdots(\alpha_j'+\alpha_j'')\cdots\alpha_k\rangle
	    	\geq \langle\alpha_1\cdots\alpha_j'\cdots\alpha_k\rangle
	    	+\langle\alpha_1\cdots\alpha_j''\cdots\alpha_k\rangle.
	    \end{equation*}
	    It coincides with the ordinary intersection product when the $\alpha_j\in\overline{\cal K}$
        are nef classes.
     \item \label{Teissier-Hovanskii ineq}
     The movable intersection product satisfies the Teissier-Hovanskii inequality
         \begin{equation*}
         \langle\alpha_1\cdot\alpha_2\cdots\alpha_n\rangle
         \geq (\langle\alpha_1^n\rangle)^{1/n}\ldots(\langle\alpha_n^n\rangle)^{1/n}.
         \end{equation*}
\end{enumerate}
\end{theorem}
It follows that the $\beta$-invariant is well-defined for any K\"ahler class. 
For any $\alpha$ in ${\cal K}$, we define
\begin{equation}
	\beta_\alpha(F)
	:=A_X(F)\Vol(\alpha)+n\mu(\alpha)\int_0^{+\infty}\Vol(\alpha-x[F])dx
	-\int_0^{+\infty}n\langle(\alpha-x[F])^{n-1}\rangle\cdot c_1(X)dx.
\end{equation}
Therefore, we can extend the valuative stability to any K\"ahler class,
\begin{definition}
\label{def of valuative stability for transc}
	For any $\alpha$ in ${\cal K}$, $(X,\alpha)$ is called 
	\begin{enumerate}[(i)]
	\item \textit{valuatively semistable} if 
          \begin{equation*}
          	\beta_\alpha(F)\geq0
          \end{equation*}
          for any prime divisor $F$ over $X$;
    \item \textit{valuatively stable} if
          \begin{equation*}
          	\beta_\alpha(F)>0
          \end{equation*}
          for any non-trivial prime divisor $F$ over $X$;
	\item \textit{uniformly valuatively stable} if there exists an $\varepsilon_\alpha>0$ such that
	      \begin{equation}
	    \label{uniform valuative stable}
		      \beta_\alpha(F)\geq\varepsilon_\alpha S_\alpha(F)
	    \end{equation}
	      for any prime divisor $F$ over $X$.
	\end{enumerate}
\end{definition}

The positive intersection product  $\langle\alpha_1\cdot\ldots\cdot\alpha_p\rangle$ depends continuously on the $p$-tuple $(\alpha_1,\ldots,\alpha_p)$ for any big classes $\alpha_1,\ldots,\alpha_p$ (see the below in \cite[Definition 1.17]{BEGZ10}).

If $\gamma$ is psef and $\alpha$ is big, by (\ref{superadditive}) and (\ref{Teissier-Hovanskii ineq}) of Theorem \ref{thm 3.5 of BDPP},  then we have
\begin{equation*}
	\Vol(\alpha+\gamma)\geq\Vol(\alpha).
\end{equation*}

A well-known result about the differentiability of the volume function on ${\cal E}^{\circ}$, due to D. Witt Nystr\"om \cite{WN19}, is stated as follows,
\begin{theorem}[{\textup{\cite[Theorem C]{WN19}}}]
On a projective manifold $X$ the volume function is continuously differentiable on the big cone ${\cal E}^{\circ}$ with
    \begin{equation}
	\label{diff of vol for big class}
		\frac{d}{dt}\Big|_{t=0}\Vol(\alpha+t\gamma)
		=n\langle\alpha^{n-1}\rangle\cdot\gamma.
	\end{equation}
	for any $\alpha$ in ${\cal E}^{\circ}$ and any $\gamma$ in $H^{1,1}(X,\bb R)$.
\end{theorem}
Therefore, we have the similar integration by part type formula:
\begin{equation}
\label{integ by part Kahler}
	\int_0^{+\infty}n\langle(\alpha-x[F])^{n-1}\rangle\cdot\alpha dx
	=(n+1)\int_0^{\infty}\Vol(\alpha-x[F])dx,
\end{equation}
for any $\alpha$ in ${\cal K}$ and any prime divisor $F$ over $X$.

It follows that these proofs of Theorem \ref{lower bdd theorem} and Theorem \ref{UVs open} can also work for the K\"ahler cone. 
In other words, the openness of uniformly valuative stability also holds on the K\"ahler cone. 
We summary as follows,
\begin{theorem}
\label{UVs open Kahler}
	For a projective manifold $X$, the uniformly valuative stability locus 
	\begin{equation*}
		\widehat{\rm UVs}:=\{\alpha\in{\cal K}|(X,\alpha)\text{ is uniformly valuatively stable}\}
	\end{equation*}
	is an open subcone of the K\"ahler cone $\cal K$.
\end{theorem}

\section{Further questions}
\label{Further questions}
In this section, we propose some further questions.
\subsection{Relation between K-stability and valuative stability}
In \cite{D-L20}, Dervan and Legendre prove that K-stability over integral test configurations is equivalent to valuative stability over dreamy divisors.
But for non-integral test configurations and non-dreamy divisors, this relation between K-stability and valuative stability is still open.
According to Fujita's idea of \cite{Fuj19 a}, we have the following setup.

Take any prime divisor $F$ over $X$, it induces a filtration $\cal F$ on
\begin{equation*}
	R:=\bigoplus_{r=0}^{\infty}R_k
	:=\bigoplus_{r=0}^{\infty}H^0(X,kL).
\end{equation*}
which is defined by
\begin{equation*}
{\cal F}^{x}R_k:=
     	\begin{cases} 
     		H^0(X,kL-\lfloor xF\rfloor) &\text{if }x\in{\bb R}_{\geq0},           \\
     		R_k  &\text{otherwise},
     	\end{cases}
\end{equation*}
for any $k\in{\bb Z}_{\geq 0}$ and $x\in\bb R$.
Then $\cal F$ is a decreasing, left continuous, multiplicative and linearly bounded $\bb R$-filtration.
We define
\begin{equation*}
	I_{(r,x)}
	:={\rm Image}({\cal F}^xR_r\otimes L^{-r}\rightarrow{\cal O}_X),
\end{equation*}
where the homomorphism is the evaluation.

Take any $e_+,e_-\in\bb Z$ with $e_+>\tau_L(F)$ and $e_-<0$, set $e:=e_+-e_-$.
Let $r_1\in\bb Z_{>0}$ be a sufficiently large positive integer.
For any $r\geq r_1$, we can define a family of flag ideals $\mathscr{I}_r\subset{\cal O}_{X\times{\bb A^1}}$ by
\begin{equation}
	\mathscr{I}_r
	:=I_{(r,re_+)}+I_{(r,re_+-1)}t^1+\cdots
	+I_{(r,re_--+1)}t^{re-1}+(t^{re}).
\end{equation}
Let $\pi_r:{\cal X}_r\rightarrow X\times{\bb A}^1$ be the blow up of $X\times {\bb A}^1$ along flag ideal $\mathscr{I}_r$, let $E_r\subset{\cal X}_{r,0}$ be the exceptional divisor, let $p_1:X\times{\bb A}^1\rightarrow X$ be the projection of the first factor of $X\times{\bb A}^1$.
We obtain a family of semiample test configurations (see \cite[Definition 2.2]{BHJ17}) $({\cal X}_r, {\cal L}_r)$, where ${\cal L}_r:=\rho_r^*L-1/rE_r$ is a relative semiample $\bb Q$-line bundle and $\rho_r=p_1\circ\pi_r$.

It is interesting to understand the limiting behavior of the Donaldson-Futaki invariant of this sequence $({\cal X}_r, {\cal L}_r)$. 
We expect to obtain one direction of the  relation between K-stability and valuative stability by studying the above limiting behavior.

We make the following question:
\begin{question}
	Does K-stability imply valuative stability?
\end{question}

\subsection{Openness of uniformly valuative stability on compact K\"ahler manifolds}

We have seen from Section \ref{Valuative stability for transcendental classes} that valuative stability is well-defined on compact K\"ahler manifolds.
But, in general, we do not know the differentiability of the volume function on the big cone ${\cal E}^{\circ}$ for compact K\"ahler manifolds. 
Then the useful formula (\ref{integ by part Kahler}) may not hold for compact K\"ahler manifolds.
In fact, in our argument, we only need the following inequality,
\begin{equation*}
C^{-1}\int_0^{\infty}\Vol(\alpha-x[F])dx\leq
	\int_0^{+\infty}\langle(\alpha-x[F])^{n-1}\rangle\cdot\alpha dx
	\leq C\int_0^{\infty}\Vol(\alpha-x[F])dx,
\end{equation*} 
for some constant $C>0$, which only depend continuously on $\alpha$.
We expect that this inequality holds. 

We make the following question:
\begin{question}
	Does the openness of uniformly valuative stability  still hold on the K\"ahler cone of compact K\"ahler manifolds?
\end{question}

\bigskip
\address{ 
	Tsinghua University,
	Yau Mathematical Sciences Center, 
	No 30, Shuangqing Road, Beijing 100084, 
	China \\
}
{\\liuyx20@mails.tsinghua.edu.cn}

\end{document}